\documentclass[12pt]{amsart}
\usepackage{amsmath}
\usepackage{amssymb}
\usepackage{latexsym}

\newcommand{\Zint}{\mathbb {Z}}    
     
\newcommand{\Cplx}{\mathbb {C}}     
\newcommand{\halmos}{\rule{5pt}{5pt}}

\numberwithin{equation}{section}

\newtheorem{prop}{\bf Proposition}[section]
\newtheorem{thm}[prop]{\bf Theorem}
\newtheorem{lemma}[prop]{\bf Lemma}

\newtheorem{cor}[prop]{\bf Corollary}
\newtheorem{conj}{\bf Conjecture}
\newtheorem{exa}{\bf Example}

\begin{document}

\title[Heun's equation, generalized hypergeometric function $\dots $]
{Heun's equation, generalized hypergeometric function and exceptional Jacobi polynomial}
\author{Kouichi Takemura}
\address{Department of Mathematics, Faculty of Science and Technology, Chuo University, 1-13-27 Kasuga, Bunkyo-ku Tokyo 112-8551, Japan.}
\email{takemura@math.chuo-u.ac.jp}
\subjclass[2000]{34M35,33E10,34M55}
\begin{abstract}
We study Heun's differential equation in the case that one of the singularities is apparent.
In particular we propose a conjecture that solutions of Heun's equation in this case also satisfy a generalized hypergeometric equation, which can be described in a more general form, and establish it in some cases.
We demonstrate application of our results to exceptional Jacobi polynomials.
\end{abstract}

\maketitle

\section{Introduction}

The hypergeometric differential equation
\begin{equation}
z(1-z) \frac{d^2y}{dz^2} + \left( \gamma - (\alpha + \beta +1)z \right) \frac{dy}{dz} -\alpha \beta  y=0,
\label{Gauss}
\end{equation}
is one of the most important differential equation in mathematics and physics.
Several properties of the hypergeometric differential equation, i.e. integral representation of solutions, explicit description of monodromy, algebraic solutions, orthogonal polynomials, etc. are studied very well, and they are applied to various problems in mathematics and physics.
The hypergeometric differential equation has three singularities $\{ 0,1,\infty \}$ and it is a canonical form of Fuchsian differential equations of second order with three singularities.
Several generalizations of the hypergeometric differential equation have been studied so far.

A generalization is given by adding regular singularities.
Heun's differential equation is a canonical form of a second-order Fuchsian equation with four singularities, which is given by
\begin{equation}
\frac{d^2y}{dz^2} + \left( \frac{\gamma}{z}+\frac{\delta }{z-1}+\frac{\epsilon}{z-t}\right) \frac{dy}{dz} +\frac{\alpha \beta z -q}{z(z-1)(z-t)} y=0,
\label{Heun}
\end{equation}
with the condition 
\begin{equation}
\gamma +\delta +\epsilon =\alpha +\beta +1, \quad t \neq 0,1. 
\label{Heuncond}
\end{equation}
Heun's differential equation appears in several systems of physics including analysis of black holes (\cite{SL}), quantum mechnics (e.g. Inozemtsev model of type $BC_1$ (\cite{Ino})), AdS/CFT correspondence (\cite{KSY}), crystal transition (\cite{SL}), fluid dynamics (\cite{CH}).
The parameter $q$ is independent from the local exponents and is called an accessory parameter.
Although it is much more difficult to study global structure of Heun's differential equation than that of the hypergeometric differential equation, several special solutions of Heun's differential equation have been investigated.
In the case that one of the regular singularities $\{ 0,1,t,\infty \}$ is apparent, the solutions of Heun's differential equation have integral representations (\cite{TakIT}), and the integral representation is based on integral transformation of Heun's differential equation established by Kazakov and Slavyanov (\cite{KS}) (see also (\cite{SL})).
Note that Shanin and Craster had studied the case of apparent singularity and obtained some other results (\cite{SC}).
On the other hand, if $\alpha \in \Zint _{\leq 0}$ and $q$ is special, then Heun's differential equation has polynomial solutions.
We will investigate the case that Heun's differential equation has polynomial solutions and the regular singularity $z=t$ is apparent.
Propositions on the structure between two conditions will be given in section \ref{sec:polapp}.

Another generalization of the hypergeometric equation is given by increasing the degree of the differential.
Let 
\begin{align}
& _p \!  F_q \left( \begin{array}{c} a_1 , a_2 , \dots ,a_p \\ b_1, b_2, \dots ,b_q \end{array} ; z \right) = \sum _{n=0}^{\infty } \frac{ (a_1)_n  (a_2 )_n \dots (a_p )_n }{ (b_1)_n (b_2)_n \dots (b_q)_n n! }z^n, 
\end{align}
be the generalized hypergeometric function, where $(a)_n =a(a+1)\dots (a+n-1)$.
Then it satisfies the generalized hypergeometric differential equation
\begin{align}
& \left\{ \frac{d}{dz} \left(  z\frac{d}{dz} +b_1 -1\right) \dots \left(  z\frac{d}{dz} +b_q -1\right) - \left(  z\frac{d}{dz} +a_1 \right) \dots \left(  z\frac{d}{dz} +a_p \right) \right\} y=0. 
\label{eq:GHGDE}
\end{align}
In the case $p=2$ and $q=1$, the function $_2 F_1 \: (\alpha , \beta ; \gamma ;z) $ is called Gauss hypergeometric function, and Eq.(\ref{eq:GHGDE}) is just the hypergeometric differential equation.
If $p=q+1$, then the differential equation (\ref{eq:GHGDE}) is Fuchsian with singularities $z=0,1, \infty$, and it is known to be rigid (\cite{Katz}), i.e. there is no accessory parameter in Eq.(\ref{eq:GHGDE}).
Consequently we have integral representations of solutions of the generalized hypergeometric differential equation.

In this paper we study some cases that the generalized hypergeometric differential equation is factorized and Heun's differential equation appears as a factorized component.
Let $L _{a_1, \dots ,a_{q+1} ; b_1 \dots b_q}$ be the monic differential operator of order $q+1$ such that $L _{a_1, \dots ,a_{q+1} ; b_1 \dots b_q}  y=0$ is equivalent to Eq.(\ref{eq:GHGDE}).
For example 
\begin{align}
& L_{a_1, a_2, a_3; b_1, b_2} = \frac{d^3}{dz^3} + \frac{(a _1+a_2 +a_3 +3)z -(b_1+b_2+1)}{z(z-1)} \frac{d^2}{dz^2} \\
& + \frac{(a _1 a_2 +a_1 a _3+ a_2 a_3 +a _1+a_2 +a_3 +1)z- b_1 b_2}{z^2(z-1)} \frac{d}{dz} + \frac{a_1a_2a_3}{z^2(z-1)} . \nonumber 
\end{align}
Letessier, Valent and Wimp (\cite{LVW}) studied generalized hypergeometric differential equations in reducible cases.
They proved that the function
\begin{equation}
{}_{p+r} F_{q+r} \left( \begin{array}{c} a_1 , \dots ,a_p , e_1 +1 , \dots , e_r+1 \\ b_1, b_2, \dots ,b_q , e_1, \dots , e_r \end{array} ; z \right) 
\end{equation}
satisfies a linear differential equation of order $\max (p,q+1)$ whose coefficients are polynomials. 
We now explain it in the case $p=2$, $q=1$, $r=1$.
Let $f_0 (z) = $ $_2 F_1 \: (\alpha , \beta ; \gamma ;z) $.
Then the function $f_1(z)= z f_0'(z)/e_1 + f_0(z)$ is equal to the function  $_{3}  F_{2} \left( \begin{array}{c} \alpha , \beta  , e_1 +1 \\ \gamma ,  e_1 \end{array} ; z \right) $ and it satisfies $L_{\alpha , \beta , e_1 +1 ; \gamma ,e_1} f_1(z) =0 $.
On the other hand, it also satisfies
\begin{align}
& \tilde{L}_{\alpha , \beta ;\gamma ; e_1 } f_1(z)=0,
 \label{eq:e1}
\end{align}
where 
\begin{align}
& \tilde{L}_{\alpha , \beta ;\gamma ; e_1 } = \frac{d^2}{dz^2} +\left( \frac{\gamma }{z} +\frac{\alpha +\beta -\gamma +2}{z-1} -\frac{1}{z-t} \right) \frac{d}{dz}  +\frac{\alpha \beta z -q}{z(z-1)(z-t)} ,\\
& t =\frac{e_1(e_1+1- \gamma )}{(e_1 -\alpha )(e_1-\beta )}, \quad q= \frac{\alpha \beta (e_1+1)(\gamma -e_1-1)}{(e_1-\alpha )(e_1-\beta )} , \nonumber
\end{align}
and we have the factorization
\begin{align}
&  L_{\alpha , \beta , e_1 +1 ; \gamma ,e_1} = \left(  \frac{d}{dz} +\frac{e_1 +1}{z} +\frac{1}{z-1} +\frac{1}{z-t} \right)  \tilde{L}_{\alpha , \beta ;\gamma ; e_1 }. 
\end{align}
Since the point $z=t $ is not singular with respect to the differential equation $L_{\alpha , \beta , e_1 +1 ; \gamma ,e_1} y =0 $, it is an apparent singularity with respect to $\tilde{L}_{\alpha , \beta ;\gamma ; e_1 } y =0 $.
Maier (\cite{Mai}) observed the fact conversely and established that Heun's differential equation with the apparent singularity $z=t$ whose exponents are $0$, $2$ appears as a right factor of the generalized hypergeometric equation $  L_{\alpha , \beta , e_1 +1 ; \gamma ,e_1}y =0$ with a suitable value $e_1$ (see Proposition \ref{prop:Maier}).
On the other hand, Heun's diffenrential equation whose singularity $z=t$ is apparent (and $\epsilon \in \Zint _{\leq -1}$) is solvable in the sense that it admits integral representations of solutions, which we explain in section \ref{sec:inttrans}.
Then it would be natural to ask that a solvale differential equation is related to rigid differential equation (i.e. the generalized hypergeometric equation).
In this paper we generalize Maier's result and propose a conjecture;
\vspace{.1in}\\
{\bf Conjecture.} (Conjecture \ref{conj}.)  {\it Set
\begin{align}
& \tilde{L}  = \frac{d^2}{dz^2} +\left( \frac{\gamma }{z} +\frac{\delta }{z-1} -\sum _{k=1}^M  \frac{m_k}{z-t_k} \right)  \frac{d}{dz} 
+\frac{s_M z^{M} + \dots + s_0 }{z(z-1)(z-t_1)\dots (z-t_M)} ,\nonumber
\end{align}
and assume that $0,1, t_1 , \dots ,t_M $ are distinct mutually, $m_1 ,\dots ,m_M \in \Zint _{\geq 1} $ and the singularities $z= t_k $ of $\tilde{L}  y=0 $ are apparent for $k=1,\dots ,M$.
Then there exists a generalized hypergeometric differential operator $L_{\alpha , \beta , e_1 +1, \dots , e_N+1 ;\gamma , e_1, \dots , e_N} $ $(N=m_1 +\dots +m_M) $ and a differential operator $\tilde{D} $ which admit the factorization}
\begin{align}
&  L_{\alpha , \beta , e_1 +1, \dots , e_N+1 ;\gamma , e_1, \dots , e_N} = \tilde{D} \tilde{L}.
\end{align}
Note that Heun's differential equation whose singularity $z=t$ is apparent and $\epsilon \in \Zint _{\leq -1}$ is included in the situation of the conjecture of the case $M=1$.
The following proposition is an evidence of supporting the conjecture;
\vspace{.1in}\\
{\bf Proposition.} (Proposition \ref{prop:conj}.)  {\it The conjecture is true for the cases $M=1$, $m_1 \leq 5$, $M=2$, $m_1+ m_2 \leq 4$ and $M=3$, $m_1=m_2=m_3=1$.}
\vspace{.1in}

Gomez-Ullate, Kamran and Milson (\cite{GKM}) introduced $X_1$-Jacobi polynomials as an orthogonal system within the Sturm-Liouville theory.
They are remarkable and are stuck out the classical framework because the sequence of polynomials starts from a polynomial of degree one.
Sasaki et al. (\cite{OS,STZ}) extended it to two types of $X_{\ell }$-Jacobi polynomials $(\ell =1,2, \dots )$ and studied properties of them.
It is known that $X_1$-Jacobi polynomials satisfy Heun's differential equation.
We apply results in this paper to $X_1$-Jacobi polynomials.
Then we may understand a position of $X_1$-Jacobi polynomials in the theory of Heun's differential equation.
Moreover we establish that $X_1$-Jacobi polynomials are also expressed by generalized hypergeometric functions.

This paper is organized as follows.
In section \ref{sec:Heunappsing}, we review definitions and properties of Heun's differential equation, apparent singularity and Heun polynomial, which will be well known for specialists.
In section \ref{sec:inttrans}, we recall an integral transformation of Heun's differential equation and its application to the case that singularity $z=t$ is apparent.
In section \ref{sec:GHGEHeun}, we give a conjecture on Heun's differential equation with an apparent singularity and reducible generalized hypergeometric equation, and verify it for some cases.
In section \ref{sec:polapp}, we explain propositions in the case that Heun's differential equation has polynomial solutions and the regular singularity $z=t$ is apparent.
In section \ref{sec:X1Jacobi}, we give applications to $X_1$-Jacobi polynomials.
In appendix, we give an outline of the proof of the conjecture in some cases.

\section{Heun's differential equation, apparent singularity and Heun polynomial} \label{sec:Heunappsing}
We review some facts on a regular singularity and Heun's differential equation, which will be applied in this paper.
\subsection{Local solution}
Let us consider local solutions of Heun's differential equation
\begin{equation}
\frac{d^2y}{dz^2} + \left( \frac{\gamma}{z}+\frac{\delta }{z-1}+\frac{\epsilon}{z-t}\right) \frac{dy}{dz} +\frac{\alpha \beta z -q}{z(z-1)(z-t)} y=0,
\label{eq:Heun}
\end{equation}
$(\gamma +\delta +\epsilon =\alpha +\beta +1)$
about $z=t$.
The exponents about $z=t$ are $0$ and $1-\epsilon $. 
If $\epsilon \not \in \Zint $, then we have a basis of local solutions about $z=t$ as follows;
\begin{align}
& f (z)= 
\displaystyle \sum _{j=0} ^{\infty } c _j (z-t)^{j}, \; (c_0 \neq 0), \quad  g(z)= 
\displaystyle (z-t) ^{1-\epsilon }  \sum _{j=0} ^{\infty } \tilde{c}_j (z-t)^{j}, \; (\tilde{c}_0 \neq 0).
\end{align}
The coefficients $c_i$ are recursively determined by $ \epsilon t (t-1) c_1 +(\alpha \beta t-q )c_0  =0  $ and 
\begin{align}
& i(i+\epsilon -1)t (t-1) c_{i} +(i+\alpha -2)(i+\beta -2)c_{i-2} 
\label{eq:recrel} \\
& +[ (i-1)(i-2)(2t - 1 )+ (i-1)\{ (\gamma +\delta +2 \epsilon )t - \gamma - \epsilon \} + \alpha \beta t -q]c_{i-1} =0, \nonumber
\end{align}
 for  $i \geq 2$ (see \cite{Ron}).
Hence $c_i$ is a polynomial of the variable $q$ of order $i$.

\subsection{Apparent singularity}
If $\epsilon \in \Zint _{\leq 0} $, then we have a basis of local solutions of Heun's differential equation as follows;
\begin{align}
& f (z)= 
\displaystyle \sum _{j=0} ^{\infty }  c _j (z-t)^{j}  + A \: g (z) \log (z-t) , \;
g (z)=
\displaystyle (z-t) ^{1-\epsilon } \sum _{j=0} ^{\infty } \tilde{c} _j(z-a)^{j}, \label{eq:locsoltaint}
\end{align}
If the logarithmic term in Eq.(\ref{eq:locsoltaint}) disappears, i.e. $A =0$, then the singularity $z=t$ is called apparent (see \cite{Ince}).
The terminology "false singularity" was used in  (\cite{SC}).
Note that the apparency of a regular singularity is equivalent to that the monodromy about $z=t$ is trivial i.e. the monodromy matrix is the unit.

Now we describe an explicit condition that the regular singularity $z=t $ of Heun's differential equation
is apparent in the case $\epsilon \in \Zint _{\leq 1}$, which was also studied in (\cite{SC}).
The condition is written as
\begin{align}
& (\alpha -\epsilon -1)(\beta -\epsilon -1)c_{ -\epsilon -1} \\
& +[ \epsilon (\epsilon +1)(2t - 1 ) -\epsilon \{ (\gamma +\delta +2 \epsilon )t - \gamma - \epsilon \} + \alpha \beta t -q]c_{ -\epsilon } =0, \nonumber
\end{align}
where $c_1 , \dots , c _{-\epsilon-1} , c _{-\epsilon}$ are determined recursively by Eq.(\ref{eq:recrel}), and the condition is nothing but the case $i=1 -\epsilon $ of Eq.(\ref{eq:recrel}) (see (\cite{Ince}) for a general theory).
By setting $n=1-\epsilon $, we have $\delta = \alpha + \beta -\gamma +n $ and the condition that the singularity $z=t $ is apparent is written as $P^{\sf app} (q)=0$,
where $P^{\sf app} (q)$ is a polynomial of the variable $q$ of order $n$, which is also a polynomial of $t,\alpha ,\beta, \gamma $.

\begin{exa} 
(i) If $\epsilon =0$ ($n=1$), then the condition that the regular singularity $z=t $ is apparent is written as $P^{\sf app} (q)= q- \alpha \beta t=0$ and it follows that the singularity $z=t$ disappears.\\
(ii) If $\epsilon =-1$ ($n=2$), then the condition that the regular singularity $z=t $ is apparent is written as 
\begin{align}
& P^{\sf app} (q)=q^2-\{(2 \alpha  \beta +\alpha +\beta ) t-\gamma +1\} q+\alpha  \beta  t \{ (\alpha +1) (\beta +1) t-\gamma \}=0 . \label{eq:ep-1q}
\end{align}
(iii) If $\epsilon =-2$ ($n=3$), then the condition that the regular singularity $z=t $ is apparent is written as 
\begin{align}
& P^{\sf app} (q)= q^3+\{ (-3 \alpha  \beta -3 \alpha -3 \beta -1)t+(3 \gamma -4) \} q^2 \label{eq:ep-2q} \\
& \qquad \quad +[ \{ 3 \alpha ^2 \beta ^2+6 \alpha  \beta (\alpha +\beta )+10 \alpha  \beta + 2 (\alpha ^2+ \beta ^2)+2 \alpha +2 \beta \} t^2 \nonumber \\
& \qquad \qquad +\{ (-6 \alpha  \beta -4  \alpha -4  \beta  )\gamma +4 \alpha  \beta +4 \alpha +4 \beta \} t+ 2 (\gamma -1) (\gamma -2) ] q \nonumber \\
& \quad -\alpha  \beta  t\{ (\alpha +1) (\alpha +2) (\beta +1) (\beta +2)t^2 -\gamma  (3 \alpha  \beta +4 \alpha +4 \beta +4)t + 2 \gamma  (\gamma -1) \}=0 .\nonumber 
\end{align}
\end{exa}

\subsection{Heun polynomial}
Heun polynomial is a polynomial (or polynomial-type) solution of Heun's differential equation (see (\cite{Ron}) and references therein).
Polynomial solutions of Fuchsian differential equation have been also studied for a long time (see (\cite{MS}) and related papers for reviews).
Here we review a condition that Heun's differential equation has a non-zero polynomial solutions of degree $N-1$.
Then the solution has an asymptotic $(1/z)^{1-N}$ as $z \rightarrow \infty $ and we have $1-N=\alpha $ or $1-N=\beta $, because the exponents about $z=\infty $ are $\alpha $ and $\beta $.
We now assume that $1-\alpha = N \in \Zint _{\geq 0} $.
If the accessory parameter $q$ satisfies $c_{N}=0$ where $c_{N}$ is determined by Eq.(\ref{eq:recrel}),
then it follows from Eq.(\ref{eq:recrel}) in the case $i=N+1$ that $c_{N+1}=0$.
Thus we have $c_i =0$ for $i \geq N+2$ and we obtain a polynomial solution of degree $N-1$.
The polynomial solution is called Heun polynomial.
Note that the condition $c_{N}=0$ is written as $P^{\sf pol} (q)=0$ by multiplying a suitable constant, where $P^{\sf pol} (q)$ is a monic polynomial of the variable $q$ of order $N$, which is also a polynomial of $t,\beta, \gamma ,\epsilon $. $(\delta =\beta -N-\gamma -\epsilon +2)$
\begin{exa}
(i) If $\alpha =0$ ($N=1$) and $\beta \not \in \Zint $, then the condition for existence of non-zero polynomial solution of Heun's equation is written as $P^{\sf pol} (q)= q=0$ and a polynomial solution is $y=1$.\\
(ii) If $\alpha =-1$ ($N=2$) and $\beta \not \in \Zint $, then the condition for existence of non-zero polynomial solution of Heun's equation is written as
\begin{align}
& P^{\sf pol} (q)= q^2+( (\beta -\epsilon ) t+\gamma  +\epsilon  )q + \beta  \gamma   t =0 , \label{eq:al-1q}
\end{align}
and a polynomial solution is
\begin{align}
& y= t (t-1) \epsilon +(q+ \beta  t) (z-t). \label{eq:al-1y}
\end{align}
(iii) If $\alpha =-2$ ($N=3$) and $\beta \not \in \Zint $, then the condition for existence of non-zero polynomial solution of Heun's equation is written as
\begin{align}
& P^{\sf pol} (q)= q^3+\{ (3 \beta -3 \epsilon -1)t+3 \gamma  +3 \epsilon +2\} q^2  \label{eq:al-2q} \\
& \qquad \qquad  \; \; \; +\{ 2 (\beta -\epsilon ) (\beta -\epsilon  -1) t^2-4(\epsilon ^2 +(\gamma - \beta  +1) \epsilon  -(2 \gamma  +1) \beta  )t \nonumber \\ 
& \qquad \qquad \qquad + 2 (\gamma + \epsilon ) (\gamma +\epsilon  +1) \} q  + 4 \beta \gamma t ((\beta  -\epsilon ) t+\gamma  +\epsilon +1)  =0 , \nonumber
\end{align}
and a polynomial solution is
\begin{align}
& y= 2 t^2 (t-1)^2 \epsilon  (\epsilon  +1) +2 t (t-1) (\epsilon  +1) (q+ 2 \beta  t) (z-t)  \label{eq:al-2y} \\
& \qquad  + \{ q^2 +((3 \beta  -\epsilon  +1) t +\gamma  +\epsilon ) q+ 2 \beta  t ((\beta +1) t+\gamma ) \} (z-t)^2. \nonumber
\end{align}
\end{exa}
Polynomial-type solutions of Heun's differential equation are written as 
\begin{align}
y=z^{\sigma _0} (z-1)^{\sigma _1} (z-t)^{\sigma _t} p(z),
\end{align}
where $p(z)$ is a polynomial, $\sigma _0 \in \{ 0, 1-\gamma \}$, $\sigma _1 \in \{ 0, 1-\delta \}$, $\sigma _t\in \{ 0, 1-\epsilon \}$.
The condition for existing a non-zero polynomial-type solution in the case $\sigma _0 =\sigma _1 =\sigma _t =0 $ is described as $\alpha \in \Zint _{\leq 0}$ (or $\beta \in \Zint _{\leq 0}$) and $P^{\sf pol} (q)=0$.
The condition for existing a polynomial-type solution in the case $\sigma _0 =1-\gamma $, $\sigma _1 =1-\delta $ and $\sigma _t =0$ is described as $\epsilon -\alpha \in \Zint _{\leq -1}$ (or $\epsilon -\beta \in \Zint _{\leq -1}$) and $\tilde{P} (q)=0 $, where  $\tilde{P} (q)$ is a polynomial of the variable $q$ of order $-\epsilon +\alpha $ (or $-\epsilon +\beta $).
Then the order of $p(z)$ is $-\epsilon +\alpha -1 $ (or $-\epsilon +\beta -1$).

\section{Integral transformation and its application} \label{sec:inttrans}

Let $p $ be an element of the Riemann sphere $\Cplx \cup \{\infty \}$ and $\gamma _p$ be a cycle on the Riemann sphere with variable $w$ which starts from a base point $w=o$, goes around $w=p$ in a counter-clockwise direction and ends at $w=o$.
Let $[\gamma _z ,\gamma _p] = \gamma _z \gamma _p \gamma _z ^{-1} \gamma _p ^{-1}$ be the Pochhammer contour.
Kazakov and Slavyanov (\cite{KS}) established that Heun's differential equation admits integral transformations.
\begin{prop} $($\cite{KS,TakMH}$)$ \label{prop:Heunint}
Set 
\begin{align}
& (\eta -\alpha )(\eta -\beta )=0 , \; \gamma '=\gamma  -\eta+1 , \;\delta ' =\delta -\eta +1, \; \epsilon '=\epsilon -\eta +1 , \label{eq:mualbe} \\
& \{ \alpha ' , \beta ' \} = \{ 2-\eta , \alpha +\beta -2\eta +1 \} , \nonumber \\
& q'=q+(1-\eta )(\epsilon +\delta  t+(\gamma  -\eta ) (t+1)). \nonumber
\end{align}
Let $v(w)$ be a solution of 
\begin{equation}
\frac{d^2v}{dw^2} + \left( \frac{\gamma  '}{w}+\frac{\delta '}{w-1}+\frac{\epsilon '}{w-t}\right) \frac{dv}{dw} +\frac{\alpha ' \beta ' w -q'}{w(w-1)(w-t)} v=0.
\label{Heun01}
\end{equation}
Then the function 
\begin{align}
& y(z)= \int _{[\gamma _z ,\gamma _p]} v(w) (z-w)^{-\eta } dw 
\label{eq:inttransHeun}
\end{align}
is a solution of 
\begin{equation}
\frac{d^2y}{dz^2} + \left( \frac{\gamma }{z}+\frac{\delta }{z-1}+\frac{\epsilon }{z-t}\right) \frac{dy}{dz} +\frac{\alpha \beta z -q}{z(z-1)(z-t)} y=0,
\label{Heun02}
\end{equation}
for $p \in \{ 0,1,t,\infty \}$.
\end{prop}

It was obtained in (\cite{TakIT}) that polynomial-type solutions of Heun's equation correspond to 
solutions which have an apparent singularity by the integral transformation.
In particular we have the following proposition by setting $\eta =\beta $ in Proposition \ref{prop:Heunint}.
\begin{prop} \label{thm:exprnonlogsol} $($\cite{TakIT}$)$ 
If $\epsilon \in \Zint _{\leq 0} $, $\alpha , \beta , \beta - \gamma , \beta -\delta \not \in \Zint $ and the singularity $z=t$ of Eq.(\ref{Heun02}) is apparent, then there exists a non-zero solution of Eq.(\ref{Heun01}) which can be written as $v(w) = w ^{\beta -\gamma }(w-1) ^{\beta -\delta } h(w)$ where $h(w)$ is a polynomial of degree $-\epsilon $ and the functions
\begin{align}
& y _p(z)  = \int _{[\gamma _z ,\gamma _p]} w ^{\beta -\gamma }(w-1) ^{\beta -\delta } h(w) (z-w)^{-\beta } dw ,
\label{eq:intp01}
\end{align}
$(p=0,1)$ are non-zero solutions of Eq.(\ref{Heun02}).
\end{prop}
We may drop the condition $\alpha , \beta , \beta - \gamma , \beta -\delta \not \in \Zint $ in Proposition \ref{thm:exprnonlogsol} by replacing to that $h(w)$ is a polynomial of degree no more than $-\epsilon $ and the solutions in Eq.(\ref{eq:intp01}) for $p=0,1$ may be zero.

\begin{cor} \label{cor:HeunHG}
If $\epsilon \in \Zint _{\leq 0} $, $\alpha , \beta , \beta - \gamma , \beta -\delta \not \in \Zint $ and the singularity $z=t$ of Eq.(\ref{Heun02}) is apparent,
then any solutions of Eq.(\ref{Heun02}) can be expressed by a finite sum of hypergeometric functions.
\end{cor}
\begin{proof}
It follow from Proposition \ref{thm:exprnonlogsol} in the case $\epsilon =0$ and $q=\alpha \beta t$ that the functions 
\begin{align}
& F_p (z)= \int _{[\gamma _z ,\gamma _p]} w ^{\beta -\gamma }(w-1) ^{\gamma - \alpha -1} (z-w)^{-\beta } dw ,
\end{align}
$(p=0,1)$ are non-zero solutions of hypergeometric differential equation, if $\alpha , \beta , \beta - \gamma , \beta -\delta \not \in \Zint $.
Since $F_0(z) \sim z^{1-\gamma }$ $(z \rightarrow 0)$ and $F_1 (z) \sim (z-1) ^{\gamma -\alpha -\beta }$ $(z \rightarrow 1)$, we have
\begin{align}
& F_0 (z) = d _{\alpha, \beta ,\gamma } z^{1-\gamma } F(\alpha -\gamma +1, \beta -\gamma +1; 2-\gamma ;z) , \\
& F_1 (z) = \tilde{d} _{\alpha, \beta ,\gamma } (1-z) ^{\gamma -\alpha -\beta } F(\gamma -\alpha ,\gamma - \beta ; \gamma -\alpha -\beta +1 ;1-z) , \nonumber 
\end{align}
where $d _{\alpha, \beta ,\gamma } ,\tilde{d} _{\alpha, \beta ,\gamma }  $ are constants.
By expanding $h(w)= \sum _{i=0}  ^{-\epsilon } c' _{i} w^i $ (resp. $h(w)= \sum _{i=0}  ^{-\epsilon } \tilde{c}' _{i} (1-w)^i $) and applying the formula,
the functions $y_0 (z)$ and $y_1(z)$ are written as a finite sum of hypergeometric functions.
If the functions $y_0 (z)$ and $y_1(z)$ are linearly dependent, then $y_0 (z) = C y_1(z)$ for some $C \in \Cplx \setminus \{ 0 \}$, $y_0^{\gamma _0} (z) = e^{-2\pi \sqrt{-1}\gamma } y_0 (z) $, $y_0 ^{\gamma _1} (z) =  e^{2\pi \sqrt{-1}(\gamma -\alpha -\beta )} y_0 (z)$ and $ y_0 ^{\gamma _{\infty }} (z) = y_0 ^{-\gamma _0 -\gamma _1} (z) = e^{2\pi \sqrt{-1}(\alpha +\beta )} y_0 (z)$, where $y_0^{\gamma } (z)$ denotes the analytic continuation along the cycle $\gamma $.
Since the exponents about $z=\infty $ are $\alpha $ and $\beta $, we have $ e^{2\pi \sqrt{-1}(\alpha +\beta )}=e^{2\pi \sqrt{-1}\alpha }$ or $e^{2\pi \sqrt{-1}\beta }$, and it contradicts to $\alpha , \beta \not \in \Zint $.
Therefore the functions $y_0 (z)$ and $y_1(z)$ are linearly dependent, every solution is written as a sum of them, and we have the corollary.
\end{proof}
We describe Proposition \ref{thm:exprnonlogsol} and Corollary \ref{cor:HeunHG} in the case $\epsilon =-2 $ explicitly.
\begin{prop} \label{prop:intrepep-2}
Set $\epsilon =-2 $.
The condition that the singularity $z=t$ of Eq.(\ref{Heun02}) is apparent is written as Eq.(\ref{eq:ep-2q}).
Then there exists a non-zero solution of Eq.(\ref{Heun01}) written as $v(w) =  w ^{\beta -\gamma }(w-1) ^{\beta -\delta }h(w)$ where
\begin{align}
& h(w)= 2 \alpha (\alpha +1) w^2 + 2(\alpha +1) \{ q- \alpha (\beta +2 ) t \} w\\
& +q^2 - \{ (2\alpha \beta +3\alpha +\beta +1)t-\gamma +2 \} q+\alpha t\{ t(\alpha +1)(\beta +1)(\beta +2)- \beta \gamma \}  , \nonumber 
\end{align}
and the functions
\begin{align}
& \int _{[\gamma _z ,\gamma _p]} w ^{\beta -\gamma }(w-1) ^{\beta -\delta } h(w) (z-w)^{-\beta } dw ,
\label{eq:intp010}
\end{align}
$(p=0,1)$ are solutions of Eq.(\ref{Heun02}).
If $\alpha , \beta , \beta - \gamma , \beta -\delta \not \in \Zint $, then the functions in Eq.(\ref{eq:intp010}) are non-zero, 
and every solution of Heun's equation with the condition that $\epsilon =-2$ and the singularity $z=t$ is apparent is written as an appropriate sum of the functions $z^{1-\gamma +k} F(\alpha -\gamma +1, \beta -\gamma +k+1; 2-\gamma +k;z)$ $(k=0,1,2 )$ and $(1-z) ^{\gamma -\alpha -\beta +k} F(\gamma -\alpha +k,\gamma - \beta ; \gamma -\alpha -\beta +k+1 ;1-z) $ $(k=0,1,2)$.
\end{prop}

\section{Generalized hypergeometric equation and Heun's differential equation with an apparent singularity} \label{sec:GHGEHeun}

We propose a conjecture on Fuchsian differential equations which have apparent singularities and generalized hypergeometric equations.
\begin{conj} \label{conj}
Set
\begin{align}
& \tilde{L}  = \frac{d^2}{dz^2} +\left( \frac{\gamma }{z} +\frac{\delta }{z-1} -\sum _{k=1}^M  \frac{m_k}{z-t_k} \right)  \frac{d}{dz} 
+\frac{s_M z^{M} + \dots + s_0 }{z(z-1)(z-t_1)\dots (z-t_M)} ,\nonumber
\end{align}
and assume that $0,1, t_1 , \dots ,t_M $ are distinct mutually, $m_1 ,\dots ,m_M \in \Zint _{\geq 1} $ and the singularities $z= t_k $ of $\tilde{L}  y=0 $ are apparent for $k=1,\dots ,M$.
Then there exists a generalized hypergeometric differential operator $L_{\alpha , \beta , e_1 +1, \dots , e_N+1 ;\gamma , e_1, \dots , e_N}  $ $(N = m_1 +m_2 + \dots + m_M)$ which admits the factorization
\begin{align}
&  L_{\alpha , \beta , e_1 +1, \dots , e_N+1 ;\gamma , e_1, \dots , e_N} = \tilde{D} \tilde{L},
\end{align}
where $\tilde{D} $ is a differential operator of order $N$ whose coefficients are rational functions.
The values $\alpha $ and $\beta $ are determined by $s_M= \alpha \beta $ and $\delta = \alpha + \beta -\gamma + N +1$.
\end{conj}
\begin{prop} \label{prop:conj}
Conjecture \ref{conj} is true for the cases $M=1$, $m_1 \leq 5$, $M=2$, $m_1+ m_2 \leq 4$ and $M=3$, $m_1=m_2=m_3=1$.
\end{prop}
We show an outline of the proof of the proposition in appendix.
Note that the calculation is dependent on Maple (a computer algebra system). 
Here we express the generalized hypergeometric differential operators of the conjecture in the cases $M=1$, $m_1 =1,2$.
Set 
\begin{align}
& H_{[\epsilon =-m ]}  = \frac{d^2}{dz^2} +\left( \frac{\gamma }{z} +\frac{\alpha +\beta -\gamma +m+1}{z-1} -\frac{m}{z-t} \right)  \frac{d}{dz}  +\frac{\alpha \beta z -q }{z(z-1)(z-t)} . 
\end{align}

The case $M=1$ and $m_1=1$ is essentially due to Maier (\cite{Mai}).
\begin{prop} $($\cite{Mai}$)$ \label{prop:Maier}
If the singularity $z= t$ of Heun's differential equation $H_{[\epsilon =-1 ]} y=0$ is apparent (see Eq.(\ref{eq:ep-1q})),
then the generalized hypergeometric differential operator $L_{\alpha , \beta , e_1 +1 ;\gamma , e_1}  $ admits the factorization
\begin{align}
&  L_{\alpha , \beta , e_1 +1 ;\gamma , e_1} = \left( \frac{d}{dz}  +\frac{e_1 +1}{z} +\frac{1}{z-1} +\frac{1}{z-t} \right)  H_{[\epsilon =-1 ]} , \\
& e_1 = \frac{q-(\alpha +1)(\beta +1)t+\gamma }{1-t}-1 . \nonumber
\end{align}
\end{prop}

Correctness of the conjecture in the cases $M=1$, $m_1 =2$ follows from the following theorem, which is verified by straightforward calculations.
\begin{thm} \label{thm:cp-2HeunGHGE}
If the singularity $z= t$ of Heun's differential equation written as $H_{[\epsilon =-2 ]} y=0$ is apparent (see Eq.(\ref{eq:ep-2q})), then there exists a generalized hypergeometric differential operator $L_{\alpha , \beta , e_1 +1 , e_2 +1 ;\gamma , e_1 ,e_2 }  $ which admits the factorization
\begin{align}
& L_{\alpha , \beta , e_1 +1 , e_2 +1 ;\gamma , e_1 ,e_2 } = \left( \frac{d^2}{dz^2}  +\left( \frac{e_1 +e_2 +3}{z} +\frac{2}{z-1} +\frac{2}{z-t} \right) \frac{d}{dz} +v(z)   \right)  H_{[\epsilon =-2 ]} ,
\end{align}
such that
\begin{align}
& v(z)= [(e_1+3) (e_2+3) z^2+\{ q -((e_1+1) (e_2+1)+(\alpha +2) (\beta +2)) t \label{eq:M1m12e1e2} \\
& \qquad -(e_1+3) (e_2+3)+2 (\gamma +1)\} z +t (e_1+1) (e_2+1)]/\{ z^2 (z-1) (z-t) \} , \nonumber \\
& e_1+e_2= -3+\frac{q-(\alpha +2) (\beta +2) t+2 \gamma }{(1-t)},\nonumber \\
&  e_1 e_2=\frac{1}{2(t-1)^2} [ q^2 - \{  (2\alpha \beta +3\alpha+3\beta  +1) t - (3\gamma -4) \}q \nonumber \\
& \qquad \qquad + ( \alpha ^2 \beta ^2 +3\alpha ^2 \beta +3\alpha \beta ^2+7\alpha \beta +2 \alpha ^2 +2\beta ^2+2\alpha  +2\beta  ) t^2 \nonumber \\
& \qquad \qquad + \{ 2\alpha \beta +4\alpha +4\beta -\gamma ( 3\alpha \beta  +4\alpha  +4 \beta)  \} t + 2(\gamma -1) (\gamma -2) ] . \nonumber 
\end{align}
\end{thm}

\section{Polynomial-type solutions with an apparent singularity} \label{sec:polapp}
If $\epsilon \in \Zint _{\leq 0}$, then the condition that $z=t$ is apparent is written as $P ^{\sf app}(q)= 0$, where $P ^{\sf app}(q)$ is  monic polynomial of $q$ with degree $1 -\epsilon $.
On the other hand, if $\alpha \in \Zint _{\leq 0}$ and $\beta \not \in \Zint _{\leq 0} $, then the condition that Eq.(\ref{Heun02}) has a polynomial solution is written as $P ^{\sf pol}(q)= 0$, where $P^{\sf pol} (q)$ is a monic polynomial of $q$ with degree $1 -\alpha $, and the degree of the polynomial solution of Eq.(\ref{Heun02}) is $-\alpha $.
In this section we investigate a relationship of equations $P ^{\sf app}(q)=0 $ and $P ^{\sf pol}(q) =0 $ in the case $\epsilon \in \Zint _{\leq 0}$ and  $\alpha \in \Zint _{\leq 0}$
\begin{lemma} \label{lem:red}
Assume that $\alpha \in \Zint $, $\epsilon \in \Zint $ and the singularity $z=t$ is apparent.
Then the monodromy representation of solutions of Eq.(\ref{Heun02}) is reducible.
\end{lemma}
\begin{proof}
Let $y_1(z)$, $y_2(z)$ be a basis of solutions of Eq.(\ref{Heun02}).
Since the singularity $z=t$ is apparent, the monodromy matrix around $z=t$ is a unit matrix.
Let $M^{(p)}$ $(p=0,1,\infty )$ be the monodromy matrix on the cycle around the singularity $w=p$ anti-clockwise with respect to the basis $y_1(z)$, $y_2(z)$.
For the moment we assume that $\gamma ,\delta \not \in \Zint$.
Then $M^{(0)}$ (resp. $M^{(1)}$) is conjugate to the diagonal matrix with eigenvalues $1$ and $e^{2\pi \sqrt{-1} \gamma } $ (resp.  $1$ and $e^{2\pi \sqrt{-1} \delta } $).
Since the exponents about $z=\infty $ are $\alpha $, $\beta $ and we have the relation $M^{(0)} M^{(1)}  =(M^{(\infty )})^{-1}$, the matrix $M^{(0)} M^{(1)} $ has an eigenvalue $1$.
Then the matrices $M^{(0)} $ and $M^{(1)} $ have an common invariant one-dimensional subspace, because if we set 
\begin{align}
M^{(0)} = \left( 
\begin{array}{cc}
1 & 0 \\
0 & g
\end{array}
\right) , \;
M^{(1)} = P \left( 
\begin{array}{cc}
1 & 0 \\
0 & d
\end{array}
\right) P^{-1}, \; 
P= \left( 
\begin{array}{cc}
p & q \\
r & s
\end{array}
\right) ,
\end{align}
the condition that $M^{(0)} M^{(1)} $ has an eigenvalue $1$ is written as $0=1-{\rm tr}(M^{(0)} M^{(1)}  ) +\det (M^{(0)} M^{(1)} )=   qr(d-1)(g-1)/(qr-ps) $, and we have an common one-dimensional eigenspace for the case $q=0$, $r=0$, $d=1$ or $g=1$ respectively.
In the case $\gamma \in \Zint $ (resp. $\delta \in \Zint $), the matrix $M^{(0)}$ (resp. $M^{(1)}$) has the multiple eigenvalue $1$, and we can also show that the matrices $M^{(0)} $ and $M^{(1)} $ have an common invariant one-dimensional subspace by expressing the matrices in the form of Jordan normal forms.
Hence the monodromy representation of solutions of Eq.(\ref{Heun02}) is reducible.
\end{proof}
Remark that Lemma \ref{lem:red} is also a consequence of the multiplicative Deligne-Simpson problem for a special case (\cite{Katz}).
\begin{prop} \label{prop:red}
Assume that $\alpha \in \Zint $, $\epsilon \in \Zint _{\leq 0}$ and the singularity $z=t$ is apparent. Set $n=-\epsilon ( \in \Zint _{\geq 0})$.\\
(i) If $\alpha >0 $, then there exists a non-zero solution $y(z)$ such that $y(z)= z ^{1-\gamma }  (z-1) ^{1-\delta } h(z) $ and $h(z)$ is a polynomial of degree no more than $\alpha +n-1$.\\
(ii) If $\alpha <1-n $, then there exists a non-zero solution $y(z)$ such that $y(z)$ is a polynomial of degree no more than $-\alpha $.\\
(iii) If $1-n \leq \alpha \leq 0$, then there exists a non-zero solution $y(z)$ such that $y(z)$ is a polynomial of degree $-\alpha $ or there exists a non-zero solution $y(z)$ such that $y(z)= z ^{1-\gamma }  (z-1) ^{1-\delta } h(z) $ and $h(z)$ is a polynomial of degree no more than $\alpha +n-1$.
\end{prop}
\begin{proof}
Assume that $\beta \not \in \Zint$, $\gamma \not \in \Zint $ and $\beta -\gamma \not \in \Zint $ for the moment.
It follows from reducibility of monodromy that there exists a non-zero solution $y(z)$ of Eq.(\ref{Heun02}) such that $y(z)= z ^{\theta _0} (z-1) ^{\theta _1} (z-t) ^{\theta _t} h(z) $ such that $h(z)$ is a polynomial, $h(0) h(1) h(t) \neq 0$, $\theta _0 \in \{ 0, 1-\gamma  \}$, $\theta _1 \in \{ 0, 1-\delta  \}$, $\theta _t \in \{ 0, 1+n \}$, and $\alpha =-\deg h(z) - \theta _0-\theta _1 -\theta _t$ or $\beta =-\deg h(z) - \theta _0-\theta _1 -\theta _t$ (see \cite[Proposition 3.1]{TakIT}).
Since $ n \in \Zint _{\geq 0}$ and $(z-t)^n$ is a polynomial in $z$, we have a non-zero solution $y(z)$ of Eq.(\ref{Heun02}) such that $y(z)= z ^{\theta _0} (z-1) ^{\theta _1} h(z) $ such that $h(z)$ is a polynomial, $h(0) h(1) \neq 0$, $\theta _0 \in \{ 0, 1-\gamma  \}$, $\theta _1 \in \{ 0, 1-\delta  \}$, and $\alpha =-\deg h(z) - \theta _0-\theta _1$ or $\beta =-\deg h(z) - \theta _0-\theta _1$.
Because $\deg h(z)$ is a non-negative integer, the possible cases under the consition $\alpha \in \Zint $, $\beta \not \in \Zint$, $\gamma \not \in \Zint $ and $\beta -\gamma = \delta -n -1- \alpha \not \in \Zint $ are the cases $\deg h(z) =-\alpha \in \Zint _{\geq 0}$ ($\alpha \leq 0 $, $(\theta _0,\theta _1) =(0,0)$) and $\deg h(z) =\alpha +n -1 \in \Zint _{\geq 0}$ ($\alpha \geq 1-n $, $(\theta _0,\theta _1) =(1-\gamma , 1-\delta )$).
Hence we have the proposition under the condition $\alpha \in \Zint $, $\beta \not \in \Zint$, $\gamma \not \in \Zint $ and $\beta -\gamma = \delta -n -1- \alpha \not \in \Zint $.

Since the monic characteristic polynomial in $q$ for existence of polynomial-type solutions $y(z)= z ^{\theta _0}  (z-1) ^{\theta _1 } h(z) $ ($h(z)$: a polynomial, $(\theta _0, \theta _1 ) =(0,0), \; (1-\gamma ,1-\delta )$) is continuous with respect to the parameters $\beta $ and $\gamma $, we obtain the proposition for all $\beta $ and $\gamma $ by continuity argument.
\end{proof}
\begin{thm} \label{thm:polapp}
Assume that $\epsilon \in \Zint _{\leq 0}$, $\alpha \in \Zint _{\leq 0}$ and $\beta \not \in \Zint _{\leq 0} $.\\
(i) If $-\alpha \leq -\epsilon $ and Heun's differential equation (Eq.(\ref{Heun02})) has a polynomial solution (i.e. the accessory parameter $q$ satisfies $P ^{\sf pol}(q)= 0$), then the singularity $z=t$ is apparent (i.e. $P ^{\sf app}(q)= 0$).\\
(ii) If $-\epsilon \leq -\alpha $ and the singularity $z=t$ is apparent (i.e. $P ^{\sf app}(q)= 0$), then Eq.(\ref{Heun02}) has a polynomial solution (i.e. $P ^{\sf pol}(q)= 0$).
\end{thm}
\begin{proof}
(ii) follows from Proposition \ref{prop:red} (ii).

We show (i).
If $\epsilon \in \Zint _{\leq 0}$, then a basis of local solutions about $z=t$ is written as
\begin{align}
& f(z)=
(z-t) ^{1-\epsilon } \sum _{j=0} ^{\infty } c _j (z-t)^{j}, 
& g(z)= 
\sum _{j=0} ^{\infty } \tilde{c}_j (z-t)^{j} 
+ A f (z) \log (z-t).
\end{align}
Apparency of the singularity $z=t$ is described as the condition $ A=0 $.
If there exists a polynomial solution $y=p(z)$ of Eq.(\ref{Heun02}), then $\deg _z p(z ) = -\alpha \leq -\epsilon $.
Since the expansion of $f(z)$ starts from $1 -\epsilon  $, the solution $p(z)$ is proportional to $g (z) $. 
Hence $ A =0 $ and we obtain (i).
\end{proof}
Proposition \ref{prop:red} and Theorem \ref{thm:polapp} are also valid for the case the singularity $z=0$ or $z=1$ is apparent.

\section{$X_1 $ Jacobi polynomial} \label{sec:X1Jacobi}

We now review a definition of $X_1$-Jacobi polynomials and their properties (\cite{GKM,OS,STZ}).
Let $P_k(\eta ) $ be the Jacobi polynomial parametrized as
\begin{equation}
P_k (\eta ) =\frac{(g+\frac{1}{2})_k }{k!} \sum_{j=0}^k \frac{(-k)_j (k+g+h+2)_j}{j! (g+\frac{1}{2})_j} \left( \frac{1-\eta }{2} \right) ^j.
\end{equation}
The $X_1$-Jacobi polynomials $\hat {P}_k (\eta ) $ $(k=0,1,2,\dots ) $ are defined in the case $g,h \not \in \{ -1/2, -3/2, -5/2, \dots \}$ by
\begin{align}
& \hat{P}_k (\eta ) = \frac{1}{k+h+\frac{1}{2}} \left( {\textstyle (h+\frac{1}{2} )} \tilde{\xi } (\eta ) P_k(\eta ) +(1+\eta ) \xi (\eta ) \frac{d}{d\eta } P_k (\eta ) \right) , \label{eq:DEX1J} \\
& \xi (\eta )= \frac{g-h}{2} \eta+\frac{g+h+1}{2}, \quad  \tilde{\xi } (\eta )= \frac{g-h}{2} \eta+\frac{g+h+3}{2}. \nonumber 
\end{align}
Hence $\deg _{\eta } \hat{P}_k (\eta ) =k+1$.
The $X_1$-Jacobi polynomials in the case $g,h >-1/2$ are orthogonal with respect to the following inner product;
\begin{align}
& \int _{-1} ^1 \hat{P}_k (\eta ) \hat{P}_{k'} (\eta ) {\mathcal W}(\eta ) d\eta = C_k \delta _{k,k'} , \quad {\mathcal W}(\eta ) = \frac{(1-\eta )^{g+\frac{1}{2}}(1+\eta )^{h+\frac{1}{2}}}{2^{g+h+2} \xi (\eta )^2} ,
\end{align} 
where $C_k$ is a non-zero constant.
The $X_1$-Jacobi polynomial $\hat {P}_k (\eta ) $ satisfies the following differential equation;
\begin{align}
& (1-\eta ^2) \frac{d^2}{d\eta ^2} \hat {P}_k (\eta ) + \left( h-g-(g+h+3)\eta -2 \frac{(1-\eta ^2) \xi '(\eta ) }{\xi (\eta )} \right) \frac{d}{d\eta } \hat {P}_k (\eta ) \\
& +\left( -\frac{2(h+\frac{1}{2})(1-\eta ) \tilde{\xi }' (\eta ) }{\xi (\eta )} +k(k+g+h+2)+g-h \right) \hat {P}_k (\eta )=0  . \nonumber 
\end{align}
By setting $\eta =1-2z$ and $y=\hat {P}_k (\eta ) $, we obtain a specific case of Heun's differential equation whose parameters are given by
\begin{align}
& \alpha  = - k-1 , \; \beta  = k + g + h + 1, \; \gamma  = g + 3/2, \; \delta = \alpha +\beta - \gamma +3 =  h  + 3/2 , \label{eq:specX1JHeun} \\
&  \epsilon = -2, \; t=\frac{1 -\gamma }{\alpha +\beta -2 \gamma +3 } =\frac{g +1/2}{g-h} , \nonumber \\ 
& q=\frac{(1 -\gamma ) (\alpha  \beta +2 \alpha +2 \beta -2 \gamma +4) }{\alpha +\beta -2 \gamma +3 } = \frac{(g +1/2)}{h-g}\{ k^2 +(g+h+2)k +g -h \}  . \nonumber
\end{align}
The condition that the singularity $z=t$ of the differential equation 
\begin{equation}
\frac{d^2y}{dz^2} + \left( \frac{\gamma }{z}+\frac{\alpha +\beta - \gamma +3 }{z-1}-\frac{2 }{z-t}\right) \frac{dy}{dz} +\frac{\alpha \beta z -q}{z(z-1)(z-t)} y=0, 
\label{Heun02ep-2}
\end{equation}
is apparent is written as Eq.(\ref{eq:ep-2q}).
By substituting $t=(1 -\gamma )/(\alpha +\beta -2 \gamma +3 )$ into Eq.(\ref{eq:ep-2q}), we have the factorization
\begin{align}
& \left( q + \frac{(\gamma -1) (\alpha  \beta +2 \alpha +2 \beta -2 \gamma +4) }{\alpha +\beta -2 \gamma +3 } \right) \cdot \\
& \left( q^2 - \frac{4\gamma ^2 -(2 \alpha  \beta +4 \alpha +4 \beta +12) \gamma +(2 \alpha  \beta +5 \alpha +5\beta + 9 )  }{\alpha +\beta -2 \gamma +3 }q \right. \nonumber \\
& \qquad \left. - \frac{\alpha  \beta ( \gamma -1) (4\gamma ^2 -( \alpha  \beta +4 \alpha +4 \beta +8) \gamma +( \alpha +1)( \beta + 1 ) }{(\alpha +\beta -2 \gamma +3 )^2} \right) =0. \nonumber
\end{align}
Hence the singularity $z=t=(1 -\gamma )/(\alpha +\beta -2 \gamma +3 )$ is apparent with respect to the second order differential equation which $X_1$-Jacobi polynomial $\hat {P}_k (1-2z ) $ satisfies. 
It follows from Proposition \ref{prop:intrepep-2} that the differenitial equation (\ref{eq:DEX1J}) admits integral representation of solutions written as 
\begin{align}
&  \int _{[\gamma _{\eta } ,\gamma _p]} \{ (k +1) (g-h) (1- \zeta ) ^2  + (2g+1)(2k+2h+3) \zeta \}  \cdot \qquad \\
&  \qquad \qquad (1 -\zeta ) ^{k + h -1/2}(1 +\zeta ) ^{k + g -1/2} (\eta -\zeta )^{-( k + g + h + 1) } d\zeta  , \nonumber 
\end{align}
where $p=-1,1$.

Next we investigate the condition that Eq.(\ref{Heun02ep-2}) has a non-zero polynomial solution under the assumption that the singularity $z=t$ is apparent (see Eq.(\ref{eq:ep-2q})).
If $\alpha = -k-1$, $\beta \neq 0$ and $k \in \Zint _{\geq 1}$, then it follows from Theorem \ref{thm:polapp} that the differential equation has a non-zero polynomial solution.
If $\alpha = -1$ $(k=0)$, then the condition that the differential equation has a non-zero polynomial solution is written as Eq.(\ref{eq:al-1q}) and we have 
\begin{align}
(q-1+\gamma ) \left( q- \frac{\beta \gamma}{\beta -2\gamma +2} \right) =0.
\end{align}
by substituting $t=(1 -\gamma )/(-1 +\beta -2 \gamma +3 )$ and $\epsilon =-2$.
On the other hand we have $q=1-\gamma$ by Eq.(\ref{eq:specX1JHeun}) in the case $k=0$ $(\alpha =-1)$.
Hence we confirm that there exists a non-zero polynomial which corresponds to the $X_1$-Jacobi polynomial in the case $k=0$.
If $\alpha = 0$, then the condition that the differential equation has a non-zero polynomial solution is written as $q=0$ and a solution is constant, and it does not agree with Eq.(\ref{eq:specX1JHeun}), i.e. $q=2(1-\gamma )(\beta -\gamma +2)/(\beta -2\gamma +3 )$.
Hence the constant does not belong to parameters of Heun's differential equation concerning to $X_1$-Jacobi polynomials.
It follows from apparency of the singularity $z=t=(1 -\gamma )/(\beta -2 \gamma +3 )$ and Proposition \ref{prop:red} that there exists a non-zero solution $y(z) $ of Heun's differential equation with the parameters in Eq.(\ref{eq:specX1JHeun}) such that $y(z)= z^{1-\gamma } (z-1)^{1-\delta } h(z)$, $\deg h(z)=1$ and the polynomial $h(z)$ is calculated as $h(z)= (\beta -2\gamma +3 )z + \gamma -2$.

It follows from Theorem \ref{thm:cp-2HeunGHGE} and apparency of the singularity $z=t$ 
that the polynomial 
$\hat {P}_k (1-2z ) $ also satisfies the generalized hypergeometric differential equation $L_{-k-1 , k+g+h+1 , e_1 +1 , e_2 +1 ;g+3/2 , e_1 ,e_2 } y=0 $, where
\begin{align}
& e_1+e_2= 2 \gamma -3 = 2g, \\
& e_1 e_2 = \frac{\alpha  \beta  (\gamma -1)}{-\gamma +2+\alpha +\beta }  = \frac{-(k+1)(k + g + h + 1)(2g+1)}{2h+1} . \nonumber
\end{align}
Thus we have the following proposition;
\begin{thm}
The $X_1$-Jacobi polynomials are expressed in terms of generalized hypergeometric functions,
\begin{align}
& \hat {P}_k (\eta ) = D_k \cdot _4 \! F _3 \left( \begin{array}{c}  - k-1, k + g + h + 1 , e_1 +1,  e_2 +1  \\ g+3/2, e_1, e_2 \end{array} ; \frac{1-\eta }{2} \right), \label{eq:X1J4F3} \\
& e_1+e_2= 2g, \quad e_1 e_2  = \frac{-(k+1)(k + g + h + 1)(2g+1)}{2h+1}, \label{eq:X1J4F3e1e2}
\end{align}
where $D_k$ is a non-zero constant. 
\end{thm}
\begin{proof}
Let $ \hat {Q}_k (\eta ) $ be the generalized hypergeometric function defined by the right hand side of Eq.(\ref{eq:X1J4F3}).
Then the functions $ \hat {Q}_k (1-2z ) $ and $\hat {P}_k (1-2z ) $ are  holomorphic solutions of the generalized hypergeometric differential equation 
\begin{equation}
L_{-k-1 , k+g+h+1 , e_1 +1 , e_2 +1 ;g+3/2 , e_1 ,e_2 } y=0 
\end{equation}
about $z=0$, where $e_1$ and $e_2$ are given by Eq.(\ref{eq:X1J4F3e1e2}).
The exponents of the differential equation about $z=0$ are $0$, $-g-1/2$, $1-e_1$ and $1 -e_2$.
If $g+1/2, e_1 , e_2 \not \in \Zint$, then the dimension of holomorphic solutions of the differential equation is one and the function $ \hat {P}_k (1-2z ) $ is proportional to $\hat {Q}_k (1-2z ) $.
By continuity argument, the function $ \hat {P}_k (1-2z ) $ is proportional to $\hat {Q}_k (1-2z ) $ in the case $g,h \not \in \{ -1/2, -3/2, -5/2, \dots \}$, the case that the functions are well-defined. 
\end{proof}

\section{Concluding remarks}
It is known that two types of $X_{\ell }$-Jacobi polynomials $(\ell =1,2, \dots )$ satisfies a second-order Fuchsian differential equation which satisfies the assumption of Conjecture \ref{conj} by setting $\eta =1-2z$  (see (\cite{OS,STZ}) etc.).
Thus relationships between $X_{\ell }$-Jacobi polynomials and generalized hypergeometric polynomials should be studied further.
On the other hand, several researchers including the authors in (\cite{AAMP,EKLW,Koo}) studied generalized Jacobi polynomials.
It would be interesting to consider relationship among those polynomials.\\

{\bf Acknowledgments}
The author would like to thank Professor Ryu Sasaki for fruitful discussions and valuable comments.
Thanks are due to Professor Satoshi Tsujimoto.
He also thanks Professor Alain Moussiaux for pointing out an error in an equation on an older version and referees for valuable comments.
He is supported by the Grant-in-Aid for Young Scientists (B) (No. 22740107) from the Japan Society for the Promotion of Science.

\appendix
\section{Sketch of the proof of Proposition \ref{prop:conj}}
We give a sketch of the proof of Proposition \ref{prop:conj}.

Let $\tilde{L} $ be the differential operator in Conjecture \ref{conj}.
Write
\begin{align}
& \tilde{D}  = \frac{d^N}{dz^N} + \sum _{j=0}^{N-1} v_{j} (z)  \frac{d^j}{dz^j} ,\quad  L_{\alpha , \beta , e_1 +1, \dots , e_N+1 ;\gamma , e_1, \dots , e_N} - \tilde{D} \tilde{L}=  \sum _{j=0}^{N+1} w_{j} (z)  \frac{d^j}{dz^j}
\end{align}
and impose the condition $w_j(z) =0$ $(j=0, \dots ,N+1)$ by choosing the coefficients $v_j (z) $ of $\tilde{D} $ appropriately.
The coefficients $v_{N-j} (z) $ $(j=1, \dots ,N )$ are determined recursively by the condition $w_{N+2-j} (z) =0$.
If we can show $w_1(z)=w_0(z)=0$ by using the condition (apparency of singularities) of the conjecture and by choosing the values of $e_1, \dots ,e_N$ appropriately, then the conjecture for the operator $\tilde{L}  $ is true.
We show it in the case $M=1$, $m_1 \leq 5$, the case $M=2$, $m_1+m_2 \leq 4$ and the case $M=3$, $m_1=m_2 =m_3=1$ with the aid of Maple.
Let ${\mathfrak e}_i$ be the $i$-th elementary symmetric function, i.e. ${\mathfrak e}_i = \sum _{1\leq j_1< \dots <j_i \leq N} e_{j_1} \dots e_{j_i}$.
The functions $w_0(z)$ and $w_1(z)$ are rational, and every coefficient of the numerators of $w_0(z)$ and $w_1(z)$ in the variable $z$ is a polynomial in ${\mathfrak e}_1, \dots ,{\mathfrak e}_N, \gamma, \delta, s_0, \dots ,s_M$, which is also dependent on $t_1, \dots ,t_M$.

We investigate the case $M=1$ and $m_1 \in \{ 1,2,3,4,5 \} $.
The differential operator $\tilde{L}  $ is written as $H_{[\epsilon =-m_1 ]} $.
Every coefficient of the numerators of $w_0(z)$ and $w_1(z)$ in $z$ is linear in ${\mathfrak e}_1, \dots ,{\mathfrak e}_N$.
We can solve the simultaneous equations determined by any coefficients the numerators of $w_1(z)$ in the variable ${\mathfrak e}_1, \dots ,{\mathfrak e}_N $.
The solution ${\mathfrak e}_1, \dots , {\mathfrak e}_N$ exists uniquely and ${\mathfrak e}_j$ ($j=1,\dots ,N $) are expressed as polynomials in $q, \alpha ,\beta , \gamma $ and are homolorphic in $t \in \Cplx \setminus \{ 0,1 \}$ (see Eq.(\ref{eq:M1m12e1e2}) in the case $m_1=2$).
By the way, apparency of the singularity $z=t$ is written as $P^{\sf app} (q)= 0 $, where $P^{\sf app} (q)$ is a polynomial of $q$ of degree $m_1+1 $. 
We evaluate the solution ${\mathfrak e}_1, \dots , {\mathfrak e}_N $ in $w_0(z)$.
Then we can confirm that all coefficients are vanished by using the relation $P^{\sf app} (q)= 0  $, and we obtain correctness of the conjecture in the case $M=1$ and $m_1 \leq 5 $.

We investigate the case $M\geq 2$. Set $N= m_1 + \dots +m_M$ and write
\begin{align}
& \tilde{L}  = \frac{d^2}{dz^2} +\left( \frac{\gamma }{z} +\frac{\alpha +\beta -\gamma + N +1}{z-1} - \sum _{k=1}^M  \frac{m_k}{z-t_k} \right)  \frac{d}{dz} \\
& \qquad \qquad +\frac{\alpha \beta  }{z(z-1)} + \sum _{k=1}^M \frac{ p_k }{z(z-1)(z-t_k)} .\nonumber 
\end{align}
The condition that the singularity $z= t_j$ $(j=1,\dots ,M) $ is apparent is written as $P^{\sf app} _j ( p_1 ,\dots ,p_M )= 0 $, where $P^{\sf app} _j ( p_1 ,\dots ,p_M) $ is a polynomial of $ p_1, \dots , p_M$ such that $\deg _{p_j} P^{\sf app} _j ( p_1 ,\dots ,p_M) =m_j +1$ and $\deg _{p_{j'}} P^{\sf app} _j ( p_1 ,\dots ,p_M ) \leq m_j $ $(j' \neq j)$. 
We assume that $M=2$, $(m_1, m_2)=(1,1)$, $(2,1)$, $(2,2)$, $(3,1)$ or $M=3$, $m_1=m_2=m_3=1 $.
We expand the numerator of $w_1(z)$ as ${\mathfrak c} _0 +{\mathfrak c} _1 (z-1)  + \dots +{\mathfrak c} _{N'} (z-1)^{N'}$ for some $N'$ and solve the simultaneous equations ${\mathfrak c} _0 = \dots = {\mathfrak c} _{N-1} =0$ for the variables ${\mathfrak e}_1, \dots , {\mathfrak e}_N$.
The solution ${\mathfrak e}_1, \dots , {\mathfrak e}_N$ exists uniquely and ${\mathfrak e}_j$ ($j=1,\dots ,N $) are expressed as polynomials in $p_1, \dots , p_M , \alpha ,\beta , \gamma $ and may have poles along $t_j (t_j -1)=0$ $(j=1,\dots ,M)$ and $t_j-t_{j'}=0$ $(1\leq j<j'\leq M)$.
We evaluate the solution ${\mathfrak e}_1, \dots , {\mathfrak e}_N $ in $w_1(z)$ and $w_0(z)$.
Then we can confirm that all coefficients are vanished by using the relation $P^{\sf app} _1 ( p_1 ,\dots ,p_M )=\dots  = P^{\sf app} _M ( p_1 , \dots ,p_M )= 0 $, and we obtain correctness of the conjecture in the case  $M=2$, $m_1+ m_2 \leq 4$ and $M=3$, $m_1=m_2=m_3=1 $.

\end{document}